\documentclass[12pt]{amsart}
\usepackage{amsmath,latexsym,amscd,amsbsy,amssymb,amsfonts,amsthm,fleqn,leqno,
euscript, graphicx, texdraw, pb-diagram, color}
\usepackage[matrix,arrow,curve]{xy}
\newtheorem{thm}{Theorem}[section]
\newtheorem{pro}[thm]{Proposition}
\newtheorem{lem}[thm]{Lemma}

\newtheorem{cor}[thm]{Corollary}

\def\con{\subseteq}
\def\from{\colon}
\def\CL#1{\overline{#1}}
\def\cl{\mathrm{cl}}
\def\leukfrac#1/#2{\leavevmode
               \kern.1em
                \raise.9ex\hbox{\the\scriptfont0 ${}_#1$}
                \hskip -1pt\kern-.1em
                /\kern-.15em\lower.10ex\hbox{\the\scriptfont0 ${}_#2$}}

\theoremstyle{definition}
\newtheorem{question}[thm]{Question}
\newtheorem{example}[thm]{Example}

\theoremstyle{remark}
\newtheorem{claim}{Claim}

\makeatletter
\def\Int{\mathop{\operator@font Int}\nolimits}
\makeatother

\begin{document}

\title[Actions of semitopological groups]
{Actions of semitopological groups}

\author{Jan van Mill}
\address{KdV Institute for Mathematics\\
University of Amsterdam\\
Science Park 105-107\\
P.O. Box 94248\\
1090 GE Amsterdam, The Netherlands}
\email{j.vanMill@uva.nl}
\thanks{The first-listed author is pleased to thank the Department of Mathematics at Nipissing University for generous hospitality and support.}

\author{Vesko Valov}
\address{Department of Computer Science and Mathematics, Nipissing University,
100 College Drive, P.O. Box 5002, North Bay, ON, P1B 8L7, Canada}
\email{veskov@nipissingu.ca}
\thanks{The second-listed author was partially supported by NSERC Grant 261914-13.}





\begin{abstract}
We investigate continuous transitive actions of semitopological groups on spaces, as well as separately continuous transitive actions of topological groups.

\end{abstract}

\maketitle
\section{Introduction}

All spaces under discussion are Tychonoff.

 Continuous actions of semitopological groups are considered in this paper. Recall that a group $G$ with a topology on the set $G$ that makes the multiplication $G\times G\to G$ separately continuous is called {\em semitopological}. A semitopological group $G$ is
{\em $\omega$-narrow} \cite{st} if for every neighborhood $U$ of the neutral element $\rm{e}$ in $G$ there is a countable set $A\subset G$ with $UA=AU=G$.

If not stated otherwise, we consider left actions $\theta:G\times X\to X$, where $G$ is a semitopological group and $X$ is a space. We denote $\theta (g,x)$ by $gx$ for all $g\in G$ and $x\in X$. If $\theta$ is continuous (resp., separately continuous), we say that the action is continuous (resp., separately continuous).
For any such an action  we consider the translations $\theta_g:X\to X$ and the maps $\theta^x:G\to X$ defined by $\theta_g(x)=gx$ and $\theta^x(g)=gx$. These two types of maps are continuous when $\theta$ is separately continuous. We also say that $\theta$ acts transitively on $X$ if all $\theta^x$, $x\in X$, are surjective maps. If the action $\theta$ can be extended to a continuous action $\widetilde\theta:G\times\widetilde X\to\widetilde X$, where $\widetilde X$ is a compactification of $X$, then $\widetilde X$ is called an {\em equivariant compactification of $X$}, or simply a $G$-compactification.

The paper is motivated by the celebrated theorem of Uspenskij \cite{u1}, \cite{u2} that a compactum $X$ is a Dugundji space provided $X$ admits a continuous transitive action of an $\omega$-narrow topological group. There is a growing interest recently in studying semitopological, quasitopological or semitopological groups versus topological groups, see \cite{ArTk}, \cite{ar}, \cite{st}.
In that direction we provide a generalization of Uspenskij's theorem in two directions. We consider spaces that are not necessarily compact
and actions not necessarily by topological groups. It is not clear to us whether our conditions are all essential, but we will show that some are, 
see Example 2.6.
A space is {\em $k$-separable} if it contains a dense $\sigma$-compact set. Our main result that we obtain by following Uspenskij's method of proof described in \cite{ArTk}, is:

\begin{thm}
Let $X$ admit a continuous transitive action 
of an $\omega$-narrow semitopological group $G$ such that $X$ has  a $G$-compactification. Then
\begin{itemize}
\item[(1)] $X$ is skeletally Dugundji provided it contains a dense \v{C}ech-complete $k$-separable subspace;
\item[(2)] $X$ is openly Dugundji if $X$ is \v{C}ech-complete and $\sigma$-compact.
\end{itemize}
\end{thm}

Recall that a continuous map $f:X\to Y$ is {\em skeletal} \cite{mr} if  $\Int \overline{f(U)}\neq\varnothing$ for every open $U\subset X$, where $\overline{f(U)}$ denotes the closure of $f(U)$ in $Y$.
Skeletally Dugundji spaces were introduced in \cite{kpv1}. In a similar way we define openly Dugundji spaces:
A space $X$ is {\em skeletally $($resp., openly$)$ Dugundji} if there exists
a well ordered inverse system
$\displaystyle S=\{X_\alpha, p^{\beta}_\alpha, \alpha<\beta<\tau\}$ with surjective skeletal (resp., open) bonding maps, where $\tau$ is a cardinal,
satisfying the following conditions: (i) $X_0$ is a separable metrizable space
and all maps $p^{\alpha+1}_\alpha$ have metrizable kernels (i.e., there exists a separable metrizable space $M_\alpha$ such that
$X_{\alpha+1}$ is embedded in $X_{\alpha}\times M_\alpha$ and $p^{\alpha+1}_\alpha$ coincides with the restriction $\pi|X_{\alpha+1}$ of
the projection $\pi\colon X_{\alpha}\times M_\alpha\to X_{\alpha}$);
(ii) for any limit cardinal $\gamma<\tau$ the space $X_\gamma$ is a (dense)
subset of $\varprojlim\{X_{\alpha},p^\beta_\alpha, \alpha<\beta<\gamma\}$;
(iii) $X$ is embedded in $\varprojlim S$ such that $p_\alpha(X)=X_\alpha$ for each $\alpha$, where
$p_\alpha\colon\varprojlim S\to X_\alpha$ is the
$\alpha$-th limit projection; (iv) for every bounded continuous
real-valued function $f$ on $\varprojlim S$ there exists $\alpha\in A$ and a continuous function $g$ on  $X_\alpha$ with $f=g\circ p_\alpha$.
The inverse system $S$ is called {\em almost continuous} if it satisfies conditions (ii), and $X$ is said to be the {\em almost limit} of $S$ if condition (iii) holds, notation $X=\mathrm{a}-\varprojlim S$.  We also say that $S$ is factorizing if it satisfies condition (iv).

  {\em  Dugundji spaces} were introduced by Pelczynski \cite{p} as the compacta $X$ such that for every embedding of $X$ in another compactum $Y$ there is a regular linear extension operator $u:C(X)\to C(Y)$ between the Banach  spaces of all continuous functions on $X$ and $Y$. It was established by Haydon \cite{ha} that
 a compactum $X$ is Dugundji iff $X$ is the limit space of a well ordered inverse system satisfying the above conditions with all $p^{\beta}_\alpha$ being open. Equivalently, $X$ is a Dugundji space iff $X$ is a compact openly Dugundji space.
 There is a tight connection between skeletally Dugundji and Dugundji spaces: $X$ is skeletally Dugundji iff every compactification of $X$ is co-absolute with
a Dugundji space, see \cite[Theorem 3.3]{kpv1}.

Since any Baire space admitting a continuous and transitive action of an $\omega$-narrow topological group $G$ has a $G$-compactification (see \cite{chko1}, \cite{u1}), we have the following

\begin{cor}
Let a space $X$ admit a continuous transitive action of an
 $\omega$-narrow topological group. Then $X$ is skeletally Dugundji $($resp., openly Dugundji$)$ provided it contains a dense \v{C}ech-complete
 $k$-separable subset $($resp., $X$ is \v{C}ech-complete and $\sigma$-compact$)$.
 \end{cor}

Note that Corollary 1.2 implies Uspenskij's theorem mentioned above.


It appears that continuous actions of semitopological groups on compact spaces can be reduced to continuous actions of topological groups. This simple observation combined with Uspenskij's theorem implies the following:

\begin{thm}
If an $\omega$-narrow  semitopological group $G$ acts continuously and transitively on a pseudocompact space $X$ and $G$ is a $k$-space, then $\beta X$ is a Dugundji space.
\end{thm}


We also consider separately continuous actions of semitopological or topological groups.
Every separately continuous left action $\theta:G\times X\to X$  generates a right action
$\Theta:G\times C(X)\to C(X)$, defined by $\Theta(g,f)(x)=f(gx)$, see Lemma 2.2 below. It is easily seen that this action is separately continuous when $C(X)$ carries the pointwise convergence topology ($C(X)$ with this topology is denoted by $C_p(X)$). We say that $\theta$
is an {\em $s$-action}, if each orbit of $\Theta$ is a separable subset of $C_p(X)$. For example, if $G$ is separable and $\theta$ is separately continuous,
then $\theta$ is an $s$-action. According to Lemmas 2.2 and 2.3, this is also true if  $\theta$ is continuous and $X$ is compact.

The conclusion of Theorem 1.1 remains true if continuity of the action is weakened to separate continuity, but requiring additionally $\theta$ to be an $s$-action.

\begin{thm}
Let $X$ admit a separately continuous transitive action $\theta$
of an $\omega$-narrow semitopological group $G$ such that $\theta$ is extendable to a separately continuous $s$-action of $G$ over a
compactification of $X$. Then $X$ is skeletally Dugundji provided it contains a dense \v{C}ech-complete $k$-separable subspace.
\end{thm}

For pseudocompact spaces we have the following analogue of  Uspenskij's result \cite[Theorem 2]{u1}.

\begin{thm}
If a pseudocompact space $X$ admits a separately continuous transitive $s$-action of an $\omega$-narrow semitopological group, then $X$ is skeletally Dugundji.
\end{thm}


Concerning separately continuous actions of topological groups we have the following fact:
\begin{pro}
If a pseudocompact $($resp., \v{C}ech-complete and $\sigma$-compact$)$ space $X$ admits a separately continuous transitive action of an $\omega$-narrow topological group, then $\beta X$ is Dugundji $($resp., $X$ is openly Dugundji$)$.
\end{pro}
The paper is organizing as follows: the proofs of Theorems 1.1 and 1.3 are given in Section 2; Section 3 contains the proofs of Theorems 1.4 - 1.5 and Proposition 1.6.

\section{Continuous actions}

Our first lemma
is a version of \cite[Proposition 10.3.1]{ArTk}. Recall that a map $f:X\to Y$ is nearly open if $f(U)\subset\Int \overline{f(U)}$ for every open $U\subset X$, see \cite{ArTk}.
\begin{lem}
Let $\theta: G\times X\to X$ be a separately continuous transitive action on a Baire space $X$.
\begin{itemize}
\item[(1)] If $G$ is an $\omega$-narrow semitopological group, then all maps $\theta^x:G\to X$, $x\in X$, are skeletal;
\item[(2)] If, in addition to $(1)$, $G$ is a topological group, then $\theta^x$ are nearly open maps.
\end{itemize}
\end{lem}

\begin{proof}
$(1)$ We first prove that $\overline{\theta^x(U)}\neq\varnothing$ for every  neighborhood $U$ of the identity $\rm e$ in $G$.
Suppose $x\in X$ and $U$ is a neighborhood of $\rm e$. Since $G$ is $\omega$-narrow, there is a countable set $A\subset G$ with $AU=G$.
It follows from the transitivity of $\theta$ that the map $\theta^x$ is surjective. Hence, $\{\theta^x(gU)=gUx:g\in A\}$ is a countable cover of $X$, so is $\{\overline{gUx}:g\in A\}$. Because $X$ is a Baire space, there is $g\in A$ with ${\rm{Int}}{~}\overline{gUx}\neq\varnothing$. Consequently, ${\rm{Int}}{~}\overline{Ux}\neq\varnothing$ (recall that the translation $\theta_g:X\to X$ is a homeomorphism).

If $U\subset G$ is a non-empty open set, we choose $g\in U$ and a neighborhood $V$ of $\rm e$ with $gV\subset U$. Then ${\rm{Int}}{~}\overline{Vx}\neq\varnothing$, so  ${\rm{Int}}{~}\overline{gVx}\neq\varnothing$. Finally, since $gVx\subset Ux$, we have
${\rm{Int}}{~}\overline{Ux}\neq\varnothing$. Therefore, $\theta^x$ is skeletal.

$(2)$ This item follows from the observation that the proof of \cite[Proposition 10.3.1]{ArTk} remains true for separately continuous actions of $\omega$-narrow groups on Baire spaces.
\end{proof}
If a semitopological group $G$ acts on a space $X$, we say that a function $f\in C^*(X)$ is {\em right-uniformly continuous} if for every $\varepsilon>0$ there is neighborhood $O$ of the neutral element $\rm e$ such that for all $g\in O$ and $x\in X$ we have $|f(x)-f(gx)|<\varepsilon$. We denote the set of all
right-uniformly continuous functions on $X$ by $C^*_{r,G}(X)$.
\begin{lem}
Let $X$ be a space and $\theta:G\times X\to X$ be a continuous action of a semitopological group $G$ on $X$. Then the right
action $\Theta:G\times C^*(X)\to C^*(X)$, $\Theta(g,f)=f\circ\theta_g$, is well defined and  each translation $\Theta_g:C^*(X)\to C^*(X)$ is a linear isometry on the Banach space $C^*(X)$. Moreover, if $X$ is compact then $C(X)=C^*_{r,G}(X)$.
\end{lem}

\begin{proof}
For every $g\in G$ and $f\in C^*(X)$ we have $\Theta(g,f)(x)=f(gx)$, $x\in X$. Because $\theta_g:X\to X$ is surjective,
$||\Theta(g,f)||=||f||$. It is obvious that each map $\Theta_g:C^*(X)\to C^*(X)$, $\Theta_g(f)=f\circ\theta_g$, is linear.
So, all $\Theta_g$ are linear isometries. Moreover, for every $g,h\in G$ and $f\in C^*(X)$ we have $\Theta(g,\Theta(h,f))=\Theta(g,f\circ\theta_h)=f\circ\theta_h\circ\theta_g$. Since
$\theta_h\circ\theta_g=\theta_{hg}$, $\Theta(g,\Theta(h,f))=f\circ\theta_{hg}=\Theta(hg,f)$. So, $\Theta$ is a right action.

Suppose $X$ is compact. Let $f\in C(X)$ and $\varepsilon>0$. For every $x\in X$ choose a neighborhood $W_x$ of $x$ in $X$ such that $|f(x')-f(x)|<\varepsilon/2$ for all $x'\in W_{x}$. Because $\theta$ is continuous, there exists a neighborhood $O_x$ of $\rm e$ in $G$ and a neighborhood $V_x\subset W_{x}$ of $x$ in $X$ such that $\theta (O_x\times V_x)\subset W_{x}$. Let $\{V_{x_i}:i=1,2,..,k\}$ be a finite subcover of the cover $\{V_x : x\in X\}$, and let $O=\bigcap_{i=1}^{k}O_{x_i}$. Then, for every $g\in O$ and $x\in X$ there is $j\leq k$
with $x\in V_{x_j}$ and
$gx\in W_{x_j}$. Hence, $|f(gx)-f(x_j)|<\varepsilon/2$ and $|f(x)-f(x_j)|<\varepsilon/2$. Consequently, $|f(gx)-f(x)|<\varepsilon$,
which completes the proof of the claim.
\end{proof}

The proof of next lemma is a slight modification of the proof of \cite[Lemma 10.3.2]{ArTk}, it is included for the sake of completeness.
\begin{lem}
Let $\theta:G\times X\to X$ be a continuous  action of an $\omega$-narrow semitopological group on a space $X$ and
$\Theta:G\times C^*(X)\to C^*(X)$ be the action from Lemma $2.2$. Then the orbit $fG=\{\Theta(g,f):g\in G\}$ is a separable subset of the Banach space $C^*(X)$ for each $f\in C^*_{r,G}(X)$.
\end{lem}

\begin{proof}
We fix $f\in C^*_{r,G}(X)$. Then for every $n\in\mathbb N$ there is a neighborhood $U_n$ of $\rm e$ with $||\Theta(g,f)-f||<1/n$ for all $g\in U_n$. Because $G$ is $\omega$-narrow, for each $n$ there exists a countable set $A_n\subset G$ such that $U_nA_n=G$. Hence, $A=\bigcup_{n=1}^\infty A_n$ is also countable and it suffices to show that $fA=\{\Theta(a,f):a\in A\}$ is dense in $fG$. To this end, let $g\in G$ and for each $n$ choose $g_n\in U_n$ and $a_n\in A_n$ with $g=g_na_n$. Then $\Theta(g,f)=\Theta_{g_na_n}(f)=\Theta_{a_n}(\Theta_{g_n}(f))$. So,
$$||\Theta(g,f)-\Theta(a_n,f)||=||\Theta_g(f)-\Theta_{a_n}(f)||=||\Theta_{a_n}(\Theta_{g_n}(f))-\Theta_{a_n}(f)||.$$ Because $\Theta_{a_n}$ is an isometry,
we obtain $||\Theta(g,f)-\Theta(a_n,f)||=||\Theta_{g_n}(f)-f||=||\Theta(g_n,f)-f||<1/n$. Therefore, every neighborhood of $\Theta(g,f)$ meets $Af$, which is as required.
\end{proof}

For any compact space $X$ let $\mathcal H(X)$ be the homeomorphism group of $X$ with the compact-open topology. It is well known \cite{are} that $\mathcal H(X)$ is a topological group. We need the following observation:

\begin{lem}
Let $X$ be a compact space admitting a continuous action $\theta:G\times X\to X$ of a semitopological group $G$. Then the homomorphism $\varphi:G\to\mathcal H(X)$, defined by $\varphi(g)=\theta_g$, is continuous.
\end{lem}

\begin{proof}
Let $K\subset X$ be compact and $U\subset X$ be open. Consider the subbasic open set $[K,U]=\{h\in\mathcal H(X): h(K)\subset U\}$ in $\mathcal H(X)$. Take an
arbitrary $g\in\varphi^{-1}([K,U])$. Then $\varphi(g)=\theta_g\in [K,U]$, hence $gK\subset U$. Consider the open set $\theta^{-1}(U)$ in $G\times X$. It contains $\{g\}\times K$. So, by the compactness of $K$, there is a neighborhood $V$ of $g$ in $G$ such that $V\times K\subset\theta^{-1}(U)$. Hence, $g'\in V$ implies $g'K\subset U$, which means that $V\subset\varphi^{-1}([K,U])$.
\end{proof}

\begin{proof}[Proof of Theorem 1.1.]
Suppose $X$ contains a dense \v{C}ech-complete $k$-separable subspace $D$. Suppose also that $\theta:G\times X\to X$ is a continuous transitive action of a semitopological $\omega$-narrow group on $X$ such that $\theta$ can be extended to a continuous action $\widetilde\theta:G\times Y\to Y$, where $Y$ is a compactification of $X$. According to Lemma 2.4, the image $\varphi(G)$ is a topological $\omega$-narrow subgroup of $\mathcal H(Y)$. Since $\mathcal H(Y)$ acts continuously on $Y$, so does $\varphi(G)$. On the other hand, $\varphi(G)$ acts transitively on $X$ because $G$ does. Therefore, we may assume that $G$ is an $\omega$-narrow topological group such that $\widetilde\theta$ is a continuous action on $Y$ and its restriction
$\theta$ is a continuous and transitive action on $X$. According to Lemma 2.2, each $f\in C(Y)$ is right-uniformly continuous. 
As above, $\widetilde\theta(g,y)$ is denoted by $gy$ for all $g\in G$ and $y\in Y$. The action $\widetilde\theta$ generates a right continuous action
$\widetilde\Theta:G\times C(Y)\to C(Y)$, defined by $\widetilde\Theta(g,f)(y)=f(gy)$ (see Lemma 2.2).

Because $D$ is $k$-separable, it contains a dense $\sigma$-compact set $Z$. Let $Z=\bigcup_{i=1}^\infty F_i$ and $Y\setminus D=\bigcup_{i=1}^\infty Y_i$, where $F_i$ and $Y_i$ are compact sets. Since $F_i\cap Y_j=\varnothing$, for each $i,j$ there is $f_{ij}\in C(Y)$ with $f_{ij}(F_i)\cap f_{ij}(Y_j)=\varnothing$. Let $C_0=\{\widetilde\Theta(g,f_{ij}):g\in G, i,j=1,2,..\}$ and
 $C(Y)\setminus C_0=\{f_\gamma:\gamma<\tau\}$, where $\tau$ is the cardinality of $C(Y)\setminus C_0$. For every cardinal $0<\alpha<\tau$ define $C_\alpha=C_0\cup\{\widetilde\Theta(g,f_\gamma):g\in G, \gamma<\alpha\}$. Then
$C(Y)=\bigcup_{0\leq\alpha<\tau}C_\alpha$ and $Y$ is embedded in $\mathbb R^{C(Y)}$ by identifying each $y\in Y$ with $\big(f(y)\big)_{f\in C(Y)}$. We consider the natural projections $p_\alpha:Y\to\mathbb R^{C_\alpha}$ for $\alpha\geq 0$  and $p^\beta_\alpha:Y_\beta\to Y_\alpha$ for $\beta>\alpha$, where $Y_\beta=p_\beta(Y)$ and $Y_\alpha=p_\alpha(Y)$. Obviously all $p_\alpha$ and $p^\beta_\alpha$ are continuous and $p_\alpha=p^\beta_\alpha\circ p_\beta$ for $\beta>\alpha$.
Observe that $p_\alpha(y)=\big(f(y)\big)_{f\in C_\alpha}$ for each $y\in Y$ and $\alpha<\tau$. Because every $f\in C(Y)$ is right-uniformly continuous, we can apply Lemma 2.3 to conclude that the orbits
$\Gamma(f)=\{\widetilde\Theta(g,f):g\in G\}$, $f\in C(Y)$, are separable subsets of $C(Y)$. Hence, $C_0$ is also a separable subset of $C(Y)$ and $p_0(Y)=Y_0$ is a metric compactum. Let $X_\alpha=p_\alpha(X)$ and $Z_\alpha=p_\alpha(Z)$ for all $\alpha\geq 0$. Because $Z$ is dense in $X$, every $Z_\alpha$ is dense in $X_\alpha$.

\begin{claim}
Every $X_\alpha$ contains a dense \v{C}ech-complete subspace, and hence it is a Baire space.
\end{claim}

Indeed, let $L_0=Y_0\setminus\bigcup_{i=1}^\infty p_0(Y_i)$ and $L=p_0^{-1}(L_0)$. Then $L_0$ is \v{C}ech-complete and we have the inclusions $Z_0\subset L_0\subset X_0$ and $Z\subset L\subset X$. Moreover, the map $p_0|L:L\to L_0$ is perfect, so $L$ is also \v{C}ech-complete. Since $C_0\subset C_\alpha$, $p_\alpha^{-1}(p_\alpha(L))=L$ and the map $p_\alpha|L:L\to L_\alpha=p_\alpha(L)$ is a perfect surjection. So, $L_\alpha$ is a dense \v{C}ech-complete subspace of $X_\alpha$.

\begin{claim}
Each $p_\alpha:Y\to Y_\alpha$ is a skeletal map.
\end{claim}

Since $C_\alpha$ is $\widetilde\Theta$-invariant (i.e., $\widetilde\Theta(G,C_\alpha)=C_\alpha$), there is an action $\widetilde\theta_\alpha:G\times Y_\alpha\to Y_\alpha$,
defined by $\widetilde\theta_\alpha(g',p_\alpha(y))=\big(f(g'y)\big)_{f\in C_\alpha}$,
which makes the diagram below commutative.
{ $$
\begin{CD}
G\times Y@>{{\widetilde\theta}}>>Y\\
@VV{id\times p_\alpha}V@VV{p_\alpha}V\\
G\times Y_\alpha @>{{\widetilde\theta_\alpha}}>>Y_\alpha
\end{CD}
$$}\\
Moreover, the restriction $\theta_\alpha=\widetilde\theta_\alpha|G\times X_\alpha$ is an action on $X_\alpha$ such that $(p_\alpha|X)\circ\theta=\theta_\alpha$.
 Let show that $\theta_\alpha$ is separately continuous. For this, we only need that $\widetilde\theta$ is separately continuous.
 Indeed, fix $g\in G$. Clearly, $(\widetilde\theta_{\alpha})g$ is continuous by commutativity of the diagram and compactness of all spaces involved. Now, fix $y_\alpha\in Y_\alpha$. Pick $z\in Y$ such that $p_\alpha(z) = y_\alpha$. Then $\widetilde\theta_\alpha^{y_\alpha}$ is equal to $p_\alpha\circ \widetilde\theta^z$ and hence is continuous being the composition of two continuous functions.

Note that, since $\theta$ is transitive on $X$, each $\theta_\alpha$ acts transitively on $X_\alpha$. To show that $p_\alpha$ is skeletal, let $U\subset Y$ be open. Since $X$ is dense in $Y$, we can fix $x\in U\cap X$.
Then the maps $\theta^x:G\to X$ and
$\theta_\alpha^y:G\to X_\alpha$ are continuous and $p_\alpha\circ\theta^x=\theta_\alpha^{y}$, where $y=p_\alpha(x)$. So,  we have the commutative diagram
$$
\xymatrix{
G\ar[r]^{\theta^x}\ar[dr]_{\theta_\alpha^y}&X\ar[d]^{p_\alpha}\\
&X_\alpha\\
}
$$
Since $X_\alpha$ is Baire space, according to Lemma 2.1(2), $\theta_\alpha^{y}$ is nearly open, and $\theta_\alpha^{y}((\theta^x)^{-1}(U\cap X))=p_\alpha(U\cap X)$ implies $p_\alpha(U\cap X)\subset \Int_{X_\alpha}\cl_{X_\alpha}p_\alpha(U\cap X)$. Hence every $p_\alpha|X:X\to X_\alpha$ is nearly open. Because $X_\alpha$ is dense in $Y_\alpha$ and
$U\cap X$ is dense in $U$, $\Int_{Y_\alpha}\cl_{Y_\alpha}p_\alpha(U)\neq\varnothing$. Thus, $p_\alpha$ is a skeletal map.

\begin{claim}
Each map $p_\alpha^{\alpha+1}$ has a metrizable kernel.
\end{claim}

Let $\Gamma(f_\alpha)$ be the orbit of $f_\alpha$ and $q_\alpha:Y\to\mathbb R^{\Gamma(f_\alpha)}$ be the projection. Since, by Lemma 2.3, $\Gamma(f_\alpha)$ is a separable subspace of $C(Y)$  and $Y$ is compact, the image $q_\alpha(Y)$ is a metric compactum.  Finally,
because $p_{\alpha+1}$ is the diagonal product $p_\alpha\triangle q_\alpha$, $p_\alpha^{\alpha+1}$ has a metrizable kernel.
Note that the inverse system $S_Y=\{Y_\alpha,p^\beta_\alpha,\alpha<\beta<\tau\}$ is well ordered  such that $Y_0$ is a metric compactum, all maps $p_\alpha$ are skeletal and $Y=\varprojlim S_Y$.
The system $S_Y$ is continuous because $C_\alpha=\bigcup_{\gamma<\alpha}C_\gamma$ for every limit cardinal $\alpha<\tau$. Moreover, it follows from our construction that $S_Y$ is factorizable.
So, $Y$ is a skeletally Dugundji space. Finally, since $X$ is dense in $Y$, the space $X$ is skeletally Dugundji as well (see \cite[Theorem 3.3]{kpv1}).

To prove Theorem 1.1(2), observed that if $X$ is \v{C}ech-complete and $\sigma$-compact, then $L_\alpha=X_\alpha$ for all $\alpha$ and the maps $p_\alpha|X:X\to X_\alpha$ are perfect and nearly open (see the proof of Claim 2). Because every closed nearly open map is open, we finally obtain that all $p_\alpha|X$ are open. Then the inverse system
 $S_X=\{X_\alpha,t^\beta_\alpha,\alpha<\beta<\tau\}$ is continuous such that $t^\beta_\alpha$ are open and perfect maps, where $t^\beta_\alpha=p^\beta_\alpha|X_\beta$. Moreover $t^{\alpha+1}_\alpha$ have metrizable kernels and $X_0$ is a Polish space. Therefore, $X$ is openly Dugundji.
\end{proof}

We will now show that Theorem 1.1 does not hold for spaces on which we do not impose extra conditions. A \emph{$P$-space} is a space in which every $G_\delta$-subset is open.

\begin{lem}
Every skeletally Dugundji $P$-space is discrete.
\end{lem}

\begin{proof}
Let $X$ be a skeletally Dugundji $P$-space, and assume that $x$ is a non-isolated point of $X$. Consider $\beta X$. It consequently contains the non-isolated $P$-point $x$. Assume that $\beta X$ is co-absolute with a Dugundji space $Z$. Then there is a compact space $Y$ which admits irreducible maps $f\from Y\to \beta X$ and $g\from Y\to Z$ (irreducible maps are surjective by definition). The set $A=f^{-1}(\{x\})$ is a nowhere dense closed $P$-set in $Y$ since $f$ is irreducible. We claim that $Y$ does not satisfy the countable chain condition. Indeed, by recursion on $\alpha < \omega_1$, we will construct a nonempty open subset $U_\alpha$ of $Y$ such that $\CL{U}_\alpha \cap (A \cup \bigcup_{\beta < \alpha} \CL{U}_\beta)$. Suppose that we defined $U_\beta$ for every $\beta < \alpha < \omega_1$. Then $V=Y\setminus \bigcup_{\beta<\alpha} \CL{U}_\beta$ is a neighborhood of $A$. Hence since $A$ is nowhere dense, there is a nonempty open subset $U_\alpha$ of $X$ such that $\CL{U}_\alpha \con V\setminus A$. Then, clearly, $U_\alpha$ is as required. It now suffices to observe that the collection $\{U_\alpha : \alpha < \omega_1\}$ witnesses the fact that $Y$ does not satisfy the countable chain condition.
But this means that $Z$ does not satisfy the countable chain condition since $g$ is irreducible. This is a contradiction since $Z$ is Dugundji and hence dyadic \cite[10.1.3]{ArTk}.
\end{proof}

\begin{example}
Let $X$ be the one-point Lindel\"offication of a discrete space of size $\omega_1$. Then $X$ is a $P$-space, and hence so is its free topological group $F(X)$ \cite[7.4.7]{ArTk}. Moreover, finite products of $X$ are Lindel\"of, hence $F(X)$ is Lindel\"of \cite[7.1.18]{ArTk}. This means that $F(X)$ is a Lindel\"of $P$-group and hence, in particular, is $\omega$-narrow. Since $X$ is not discrete, $F(X)$ has no isolated points. By Teleman~\cite{tel}, $F(X)$ has an $F(X)$-compactification. But $F(X)$ is not skeletally Dugundji by the lemma just proved.
\end{example}
The following questions remain open.

\begin{question}
Let the ($k$-separable) Baire space $X$ admit a continuous transitive action
of an $\omega$-narrow semitopological group $G$ such that $X$ has  a $G$-compactification. Is $X$ skeletally Dugundji?
\end{question}

\begin{proof}[Proof of Theorem $1.3$]
Suppose $X$ is a pseudocompact space and $\theta:G\times X\to X$ is a continuous transitive action,
where $G$ is an  $\omega$-narrow semitopological group which is a $k$-space. Then,  there is a continuous action
$\widetilde\theta:G\times\beta X\to\beta X$ extending $\theta$. Indeed, since each $\theta_g:X\to X$ is a homeomorphism, $\theta_g$ can be extended to a homeomorphism $\widetilde\theta_g:\beta X\to\beta X$. Hence, we may define  $\widetilde\theta:G\times\beta X\to\beta X$ by
$\widetilde\theta(g,x)=\widetilde\theta_g(x)$. Obviously, $\widetilde\theta(g_1,\widetilde\theta(g_2,x))=\widetilde\theta(g_1g_2,x)$.
So, it remains to show that $\widetilde\theta$ is continuous. To this end, observe that $G\times\beta X$ is a $k$-space as a product of the $k$-space $G$ and the compactum $\beta X$ (see \cite[Theorem 3.3.27]{en}). Therefore, it is enough to show that each restriction $\widetilde\theta_K=\widetilde\theta|(K\times\beta X)$ is continuous, where $K\subset G$ is compact. And this is true because by \cite[Theorem 3.10.26]{en} the product $K\times X$ is pseudocompact. Hence, by Glicksberg's theorem \cite[3.12.20(c)]{en}, $\beta(K\times X)=K\times\beta X$. Thus,
$\theta|(K\times X)$ has a continuous extension to $K\times\beta X$ and it is obvious that this extension coincides with $\widetilde\theta_K$.

Now we can complete the proof of Theorem 1.3.
We apply Lemma 2.4 to the action $\widetilde\theta$ and obtain that $\varphi:G\to\mathcal H(\beta X)$ is a continuous homomorphism. So, $\varphi(G)$
is a an $\omega$-narrow topological group acting continuously on $\beta X$. Denote this action by $\widetilde\theta_\varphi:\varphi(G)\times\beta X\to\beta X$. Obviously $\widetilde\theta_\varphi$,  restricted on $\varphi(G)\times X$, provides a continuous and transitive action $\theta_\varphi:\varphi(G)\times X\to X$ on $X$. Therefore, $X$ is a pseudocompact space admitting a continuous and transitive action of an $\omega$-narrow topological group.
Hence, we can apply \cite[Theorem 2]{u1} to conclude that $\beta X$ is a Dugundji space.
\end{proof}

\section{Separately continuous actions}
First we prove the following lemma:
\begin{lem}
Let $F\subset C_p(X)$ be a separable subset and $p:X\to\mathbb R^F$ be the diagonal product of all $h\in F$. Then $p(X)$ is a sub-metrizable space.
\end{lem}

\begin{proof}
 Take a countable dense subset $\Gamma$ of $F$ and let $\varphi:X\to\mathbb R^\Gamma$ be the diagonal product of all $h\in\Gamma$. Then $\varphi (X)$ is a metrizable space. There is a continuous surjection $\lambda:p(X)\to\varphi(X)$ assigning to each point
$p(x)=(h(x)_{h\in F})\in p(X)$ the point $\varphi(x)=(h(x)_{h\in\Gamma})\in\varphi(X)$. We claim that $\lambda$ is one-to-one. Suppose $\lambda(p(x_1))=\lambda(p(x_2))$ for the distinct points $p(x_1), p(x_2)\in p(X)$. So, there is $h\in F$ with $h(x_1)\neq h(x_2)$, and let $|h(x_1)-h(x_2)|=\eta$. Since $\Gamma$ is dense in $F$ with respect to the pointwise topology, there exists $f\in\Gamma$ such that $|f(x_i)-h(x_i)|<\eta/2$ for $i=1,2$. Hence, $f(x_1)\neq f(x_2)$, which contradicts the equality $\varphi(x_1)=\varphi(x_2)$.
\end{proof}

{\em Proof of Theorem $1.4$.}
We follow the proof of Theorem 1.1. Let $X$ contain a dense \v{C}ech-complete $k$-separable space $D$ and $\theta:G\times X\to X$ be 
a transitive separately continuous action of an $\omega$-narrow semitopological group such that $\theta$ can be extended to a separately continuous $s$-action
$\widetilde\theta:G\times Y\to Y$, where $Y$ is a compactification of $X$. Suppose also that $Z$ is a dense $\sigma$-compact set of $D$. The action $\widetilde\theta$ generates a right separately continuous action $\widetilde\Theta:G\times C_p(Y)\to C_p(Y)$ such that each orbit
$\Gamma(f)=\{\widetilde\Theta(g,f):g\in G\}$, $f\in C(Y)$, is a separable subset of $C_p(Y)$. As in the proof of Theorem 1.1, we embed $Y$ in $\mathbb R^{C(Y)}$, define the sets $C_0$ and $C_\alpha$ and consider the projections $p_\alpha:Y\to\mathbb R^{C_\alpha}$, $\alpha\geq 0$. According to Claim 1, each
$X_\alpha$ is a Baire space.
Because $\widetilde\theta$ is separately continuous and $\theta$ is transitive, the actions $\theta_\alpha:G\times X_\alpha\to X_\alpha$ are also separately continuous and transitive, see Claim 2. The proof that $p_\alpha$ are skeletal maps is the same. The only difference is, according to Lemma 2.1(1), that the map $\theta_\alpha^y:G\to X_\alpha$ is skeletal for every $y\in X_\alpha$. So, for every open $U\subset Y$ and points $x\in U\cap X$ and $y=p_\alpha(y)$, we have
$\rm{Int}_{X_\alpha}\rm{cl}_{X_\alpha}\theta_\alpha^y((\theta^x)^{-1}(U\cap X))=\rm{Int}_{X_\alpha}\rm{cl}_{X_\alpha}p_\alpha(U\cap X)\neq\varnothing$. This implies that $p_\alpha$ is skeletal, see the proof of Claim 2. Finally, in the proof of Claim 3 we can use Lemma 3.1  to conclude that every $p_\alpha^{\alpha+1}$ has a metrizable kernel. $\Box$

\begin{proof}[Proof of Theorem 1.5]
Suppose $X$ is a pseudocompact space, $G$ is an $\omega$-narrow semitopological group and
$\theta:G\times X\to X$ is a separately continuous transitive $s$-action. Then
$\theta$ generates a separately continuous action $\Theta:G\times C_p(X)\to C_p(X)$ such that each orbit $\Gamma(f)=\{\Theta(g,f):g\in G\}$, $f\in C(Y)$, is a separable subset of $C_p(Y)$.
 We follow the proof of Theorem 1.1, considering $X$ instead of $Y$. We embed $X$ in $\mathbb R^{C(X)}$, where $C(X)=\{f_\gamma:\gamma<\tau\}$, and let $C_\alpha=C_0\cup\{\Theta(g,f_\gamma):g\in G, \gamma<\alpha\}$ with $C_0$ being the set of all constant functions on $X$. We also consider the projections
 $p_\alpha:X\to X_\alpha\subset\mathbb R^{C_\alpha}$, and
the transitive actions $\theta_\alpha:G\times X_\alpha\to X_\alpha$, $0\leq\alpha<\tau$. Observe that $X_0$ is a point because all $f\in C_0$ are constant functions.
Since $\theta$ is separately continuous, so is each $\theta_\alpha$ (see the proof of Claim 2). Because every $X_\alpha$ is a Baire space (being pseudocompact), following the proof of Claim 2 and using Lemma 2.1(1), we show that
all maps $p_\alpha$ are skeletal.

It remains to show that the
maps $p^{\alpha+1}_\alpha$ have metrizable kernels.
Let $\Gamma(f_\alpha)=\{\Theta(g,f_\alpha):g\in G\}$ be the orbit of $f_\alpha$ and $q_\alpha:X\to\mathbb R^{\Gamma(f_\alpha)}$ be the projection. Since $\Gamma(f_\alpha)$ is a separable subset of $C_p(X)$, by Lemma 3.1,
$q_\alpha(X)$ is sub-metrizable. On the other hand, $q_\alpha(X)$ is pseudocompact. So, by \cite{vel} $q_\alpha(X)$ is a metric compactum.
Therefore, the inverse system $S=\{X_\alpha,p^\beta_\alpha,\alpha<\beta<\tau\}$ is well ordered, almost continuous, consists of skeletal maps with $X_0$ being a point and each $p_\alpha^{\alpha+1}$ having a metrizable kernel. Moreover, $S$ is factorizable and
$X=\mathrm{a}-\varprojlim S$. Hence, $X$ is a skeletally Dugundji space.
\end{proof}

\begin{proof}[Proof of Proposition 1.6]
Suppose $\theta:G\times X\to X$ is a separately continuous transitive action on a Baire space $X$, where $G$ is an $\omega$-narrow topological group. Then, by Lemma 2.1(2), each $\theta^x:G\to X$ is nearly open. Consequently, by \cite[Proposition 2]{chko}, $\theta$ is continuous. If $X$ is \v{C}ech-complete and $\sigma$-compact, Corollary 1.2 implies that $X$ is openly Dugundji.
 In case $X$ is pseudocompact, we apply \cite[Theorem 2]{u1} to conclude that $\beta X$ is Dugundji.
 \end{proof}

\bibliographystyle{amsplain}

\end{document}